\documentclass[12pt]{amsart}
\usepackage{amsmath,amsthm,amssymb}
\newtheorem{theorem}{Theorem}[section]
\newtheorem{lemma}[theorem]{Lemma}
\newtheorem{corollary}[theorem]{Corollary}
\newtheorem{remark}[theorem]{Remark}

\numberwithin{equation}{section}

\begin{document}

\title[Fractional smoothness]{Fractional smoothness
of images of logarithmically concave measures under polynomials}

\thanks{This work has been supported by the Russian Science Foundation Grant 14-11-00196 at
Lomonosov Moscow State University.}

\author{Egor D. Kosov}

\maketitle

\begin{abstract}
We show that a measure on the real line
that is the image of a log-concave
measure under a polynomial of degree $d$ possesses
a density from the Nikol'skii--Besov class of fractional order~$1/d$. This result
is used to prove an estimate of the total variation distance
between such measures in terms of the Fortet--Mourier distance.
\end{abstract}



\section*{Introduction}

Many fundamental problems of stochastic calculus involve
investigation of the smoothness properties
of measures of the form $\nu=\mu\circ f^{-1}$, i.e., measures induced by
$\mu$-measurable functions $f$ with respect to a given
 measure $\mu$ on an infinite-dimensional
  space (e.g., the distribution of a stochastic process).
In this paper we study the class of such measures $\nu$
induced by polynomials on spaces with logarithmically concave measures.
Since all Gaussian measures are logarithmically concave,
our results apply to Gaussian measures (e.g., to the  Wiener measure).
This class of distributions $\nu$ is of interest for many
applications, because it contains typical statistics
and the class of all polynomials of a fixed degree can be considered
as an important family of nonlinear transformations of a given measure.
Various properties of measures in the class under consideration
have been studied in many works,
see \cite{BKZ, BKNP, BZ, HLN, Nourdin1, Nourdin2, NualPec, Shigek}
for the case of Gaussian measures and
\cite{ArutKos, BobkPoly, Bobk, CarWr, NSV, Nourdin3}
for the case of general logarithmically concave measures.

Our first  main result states that the density
of a polynomial image  of a log-concave measure always  belongs to
the Nikol'skii--Besov
class $B_{1,\infty}^{1/d}$ (see \cite{BIN,Nikol77}, sometimes it is also denoted by
$\Lambda^{1,\infty}_{1/d}$, see~\cite{Stein}),
where $d$ is the degree of the polynomial. We also prove the following quantitative estimate (Corollary \ref{c5.1}):
$$
\sigma_f^{1/d}
\int_\mathbb{R}|\rho(t+h) - \rho(t)|dx\le C(d)|h|^{1/d} \quad \forall h\in \mathbb{R}.
$$
Here $\rho$ is the density of the measure $\mu\circ f^{-1}$, $\mu$ is a log-concave measure, $f$ is a polynomial of degree $d$, and
$\sigma_f^2$ is the variance of $f$.

This result is used to obtain an estimate of the total variation distance
between distributions of polynomials in terms
of the Fortet--Mourier distance (Corollary \ref{c5.3}):
$$
\|\mu\circ f^{-1}-\mu\circ g^{-1}\|_{\rm TV}\le
C(d, a)\|\mu\circ f^{-1} - \mu\circ g^{-1}\|_{\rm FM}^{1/(1+d)},
$$
provided that $\sigma_f, \sigma_g \ge a$.
This estimate generalizes some recent results from
\cite{Nourdin2, Nourdin3} and  \cite{BKZ}
to the case of log-concave measures.
However,  even in the case of a Gaussian measure
the power at the Fortet--Mourier distance in our estimate is better in comparison
with similar results from the cited papers.

The paper is organized in the following way.
In Section~1 we give necessary definitions
and some preliminary results needed in the proofs of the main results.
The subsequent four sections contain the proofs of our results.
An important tool in our approach is the so-called localization
technique (Theorem \ref{t1.5}) that
enables us to reduce certain
high-dimensional inequalities to inequalities in low dimensions.
This means that if we want to obtain a dimension-free estimate for
the class of log-concave measures,
we can prove a low-dimensional estimate and then use the
localization techniques to make it dimension-free.
Let us outline some key steps in each section.
The main tool of studying smoothness of induced distributions is the Malliavin method,
 but in our case it cannot be applied directly, since the density
 of a polynomial distribution
need not be even bounded (e.g., take the $\chi^2$-distribution with one degree of freedom).
To overcome this difficulty in Section~2 we obtain a sufficient Malliavin-type
condition for the density of a measure on the real line to belong
to the Nikol'skii--Besov class (Lemma \ref{lem2.1}) and
from this we deduced an estimate of the total variation distance
in terms of the Fortet--Mourier distance
(Lemma \ref{lem2.2}).
An  important ingredient of the classical Malliavin method
is a ceratin smoothness requirement on
the measure $\mu$ the images of which we study.
Since our approach does not allow to deal with
each fixed log-concave measure, but instead deals with the whole class
of log-concave measures,
we have to provide an estimate on the derivatives of a log-concave
measure that does not depend much
on the measure. We obtain such an estimate in Section~3, where
we prove an estimate
on the variation of the Skorohod derivative of an isotropic log-concave
measure on $\mathbb{R}^n$ in terms of its isotropic constant.
In Section~4 we combine the results of the previous section
with the localization techniques
to verify our Malliavin-type condition from Section~2
for the polynomial images of log-concave measures
(Theorem \ref{t4.1}).
Note that another important ingredient
of the classical Malliavin method is a certain nondegeneracy condition imposed
on the mapping $f$ that induces the distribution under consideration.
In our case when $f$ is a polynomial,
we automatically have such a nondegeneracy condition in the
form of the Carbery--Wright inequality (Theorem \ref{t1.3}).
Finally, in
Section~5 we present our main results
(Corollaries \ref{c5.1}, \ref{c5.2}, \ref{c5.3}, and \ref{c5.4})
for log-concave measures on infinite dimensional locally convex spaces that follow from
the technical result of Theorem \ref{t4.1} and an approximation argument.

\section{Preliminaries}

In this section we introduce necessary definitions and notation.
We also formulate here some auxiliary results.

For $x,y\in \mathbb{R}^n$ let $(x, y)$
denote the standard Euclidian inner product in $\mathbb{R}^n$ and
let $|x|$ be the norm generated by this inner product, i.e. $|x|:= \sqrt{(x,x)}$.
Let $C_0^\infty(\mathbb{R}^n)$ denote the space
of all smooth functions with compact support and
let $C_b^\infty(\mathbb{R}^n)$ denote the space of all bounded
smooth functions with bounded derivatives of all orders.
For a function $\varphi$ on the real line we set
$$
\|\varphi\|_\infty:=\sup_t|\varphi(t)|.
$$

The total variation and the Fortet--Mourier distances between two probability measures
$\nu_1$ and $\nu_2$ on $\mathbb{R}$ are defined by the following equalities, respectively:
$$
\|\nu_1-\nu_2\|_{\rm TV} := \sup\biggl\{\int \varphi d(\nu_1-\nu_2), \ \varphi\in C_b^\infty(\mathbb{R}^n),\ \|\varphi\|_\infty\le1\biggr\},
$$
$$
\|\nu_1-\nu_2\|_{\rm FM} := \sup\biggl\{\int \varphi d(\nu_1-\nu_2), \ \varphi\in C_b^\infty(\mathbb{R}^n),\ \|\varphi\|_\infty\le1,\ \|\varphi'\|_\infty\le1\biggr\}.
$$
Note that $\|\nu_1-\nu_2\|_{\rm FM}\le2$.

A probability Borel measure $\mu$ on $\mathbb{R}^n$ is called
logarithmically concave (log-concave or convex) if it has
a density of the form $e^{-V}$ with respect to Lebesgue measure
on some affine subspace~$L$,
where $V\colon\, L\to (-\infty, +\infty]$ is a convex function.
This definition is equivalent to the property that for
every pair of Borel sets $A,B$ the following inequality holds
(see \cite{Bor}):
$$
    \mu(t A + (1-t)B) \ge \mu(A)^{t}\mu(B)^{1-t} \quad \forall\, t\in [0,1].
$$

A  Radon probability  measure $\mu$ on a locally convex space $E$ is called
log-concave (or convex) if
its images under  continuous  linear operators to $\mathbb{R}^n$
are log-concave.

Let $K$ be a convex body in $\mathbb{R}^n$ with $0\in K$. Its
Minkowski functional   $\|\cdot\|_K$ is defined by
$$
\|x\|_K := \inf\{t>0\colon t^{-1}x\in K\}.
$$
By convex bodies we mean closed convex sets with non empty interior.

Let $\mathrm{I}_A$ denote the indicator function of the set $A$ and let $|A|$ denote the
Lebesgue volume of the set $A$.
The symbol $\lambda$ denotes the standard Lebesgue measure on the real line.

A log-concave measure $\mu$ on $\mathbb{R}^n$ is called isotropic
if it is absolutely continuous with respect to Lebesgue measure and
$$
\int_{\mathbb{R}^n}(x,\theta)\mu(dx)=0,\quad
\int_{\mathbb{R}^n}(x,\theta)^2\mu(dx)=
L_{\mu}^2|\theta|^2\quad \forall\ \theta\in\mathbb{R}^n,
$$
where the constant $L_\mu$ is called the isotropic constant of the measure $\mu$.

Let $\mu$ be a Borel probability measure on $\mathbb{R}^n$ and let $h\in\mathbb{R}^n$.
The Skorohod derivative $D_h\mu$ of the measure $\mu$ along $h$
is a bounded signed Borel measure on $\mathbb{R}^n$ such that
$$
\int_X \partial_h \varphi(x) \mu(dx)=-\int_X \varphi(x) D_h\mu(dx)
$$
for every $\varphi\in C_b^\infty(\mathbb{R}^n)$ (see \cite{DiffMeas}).
A measure $\mu$ has the Skorohod derivatives along all vectors precisely when
it possesses a density of class $BV$ (the class of functions of bounded variation).

According to Krugova's result \cite{Krug}
(see also \cite[Section~4.3]{DiffMeas}),
for every log-concave measure $\mu$ with a density $\rho$
and for every vector $h$ of unit length one has the following equality:
$$
    \|D_h\mu\|_{\rm TV}=2\int_{\langle h\rangle^\bot} \max\limits_t\rho(x+th)dx,
$$
where $\langle h\rangle^\bot$ is the orthogonal complement of $h$.

Let $\nu$ be a Borel probability  measure on the real line
and let  $\nu_h$ denote its shift by the vector $h$:
$$
\nu_h(A):=\nu(A-h).
$$

Let $\mu$ be a Radon probability  measure on a locally convex space $E$.
Denote by
$\mathcal{P}^d(\mu)$ the closure in $L^2(\mu)$ of the set of functions of the form
$f(\ell_1,\ldots,\ell_n)$, where
$\ell_i\in E^*$ (the topological
dual space to $E$) and $f$ is a polynomial on $\mathbb{R}^n$ of degree $d$.
It is shown in \cite{ArYar} that every function from $\mathcal{P}^d(\mu)$
has a version that is a polynomial of degree $d$ in the usual algebraic sense.

For a $\mu$-measurable function $f$, let
$$
\|f\|_{r}=\biggl(\int |f|^r d\mu\biggr)^{1/r}\ \text{for} \ r>0, \quad
\|f\|_0=\exp\biggl(\int \ln|f|d\mu\biggr)
=\lim\limits_{r\to 0} \|f\|_r,
$$
$$
\mathbb{E}f:=\int f d\mu\ \text{is the expectation of the random variable } f,
$$
$$
\sigma^2_f = \int (f - \mathbb{E}f)^2d\mu\ \text{is the variance of the random variable } f.
$$

We recall (see \cite{BIN,Nikol77,Stein})
that a function $\rho\in L^1(\mathbb{R})$
belongs to the Nikol'skii--Besov class $B_{1,\infty}^\alpha$ ($0<\alpha<1$)
if there is a constant $C>0$ such that
$$
\int|\rho(x+h) - \rho(x)|dx\le C|h|^\alpha,\quad \forall h\in \mathbb{R}.
$$

\vskip .1in

The following known results will be used in the proofs.

\begin{theorem}[see \cite{Klartag, Ball}]\label{t1.1}
For every $n\in \mathbb{N}$, there is a constant $C_n$ depending only on  $n$ such that
for every isotropic log-concave measure $\mu$ on $\mathbb{R}^n$
with a density $\rho$
and the isotropic constant $L_\mu$
one has
$$
(\max\rho)^{1/n}L_\mu\le C_n.
$$
\end{theorem}

There is an open conjecture
(the hyperplane conjecture) that the constant in the previous theorem
can be chosen independent of $n$, but
for now the best known constant $C_n\sim n^{1/4}$ is due to Klartag (see \cite{Klartag}).

\begin{theorem}[see \cite{BobkPoly, Bobk}]\label{t1.2}
There is an absolute constant $c$ such that for every log-concave measure
$\mu$ on $\mathbb{R}^n$, for every number $q\ge1$ and for every polynomial
$f$ of degree $d$ on $\mathbb{R}^n$ the following inequality holds:
$$
\|f\|_q\le (cqd)^d\|f\|_p,\ \text{whenever }\ 0\le p< q<\infty.
$$
\end{theorem}

The following estimate is
usually called the Carbery--Wright inequality (see \cite{CarWr},
it was also implicitly proved in \cite{NSV}).

\begin{theorem}[see \cite{CarWr, NSV}]\label{t1.3}
There is an absolute constant $c_1$ such that for every log-concave measure
$\mu$ on $\mathbb{R}^n$ and for every polynomial
$f$ of degree $d$ the following inequality holds:
$$
\mu(|f|\le t)\biggl(\int|f|d\mu\biggr)^{1/d}\le t^{1/d} c_1d.
$$
\end{theorem}

Some analogues of the previous two theorems for measurable
polynomials on  infinite-dimen\-sio\-nal spaces are discussed in \cite{ArutKos}.

The following assertion is the Poincar\'e inequality for log-concave measures.

\begin{theorem}[see \cite{BobkIsop, KLS}]\label{t1.4}
There is an absolute constant $R$ such that for every log-concave measure
$\mu$ on $\mathbb{R}^n$ and for every locally Lipschitz  function $f$ on $\mathbb{R}^n$
one has
$$
\int\bigl(f-\mathbb{E}f\bigr)^2d\mu
\le R\int|x-x_0|^2d\mu\int|\nabla f|^2d\mu,\ \text{where }\ x_0=\int x d\mu.
$$
\end{theorem}

We also need the following result of Fradelizi and Gu\'edon
that generalizes the localization lemma from \cite{LS, KLS}.

\begin{theorem}[Localization lemma with $p$ constraints, see \cite{FrGue}]\label{t1.5}
Let $K$ be a compact convex set in $\mathbb{R}^n$,
$f_i\colon K\to \mathbb{R}$, $1\le i\le p$.
Assume that all functions $f_i$ are upper semi-continuous.
Let $P_{f_1,\ldots, f_p}$ be the set of all log-concave measures with support in $K$
such that
$$
\int f_id\mu\ge0, \ i=1,\ldots,p.
$$
Let $\Phi\colon P(K)\to\mathbb{R}$ be a convex
upper semi-continuous function, where $P(K)$
is the space of all Borel probability  measures on $K$
equipped with the weak topology.
Then $\sup\limits_{\mu\in P_{f_1,\ldots, f_p}}\Phi(\mu)$
is attained  on  log-concave measures $\mu$
such that the smallest affine subspace containing the support of $\mu$
is of dimension not greater than $p$.
\end{theorem}

\section{Sufficient conditions for fractional smoothness of measures on the real line}

In this section we provide a Malliavin-type condition for
the density of a measure on the real line to belong
to the Nikol'skii--Besov class.
We also prove several estimates for different distances between
measures satisfying this condition.
In this section we assume that the parameter $\alpha$ belongs to $(0, 1]$.

The proofs of Lemma \ref{lem2.1} and Lemma \ref{lem2.2} below
can be found in \cite{BKZ}.

\begin{lemma}\label{lem2.1}
Let $\nu$ be a probability Borel measure on the real line.
Assume that for every function
$\varphi\in C_b^\infty(\mathbb{R})$ with $\|\varphi\|_\infty\le1$
one has
$$
\int\varphi'd\nu\le C\|\varphi'\|_\infty^{1-\alpha}.
$$
Then
$$
\|\nu_h-\nu\|_{\rm TV}\le 2^{1-\alpha}C|h|^\alpha\quad \forall\, h\in\mathbb{R}.
$$
\end{lemma}

\begin{remark}\label{rem2.1}{\rm
Note that the condition
$$
\|\nu_h-\nu\|_{\rm TV}\le C|h|^\alpha
$$
is equivalent to the property that the measure
$\nu$ possesses a density from the Nikol'skii--Besov class
$B_{1,\infty}^\alpha$.}
\end{remark}

\begin{lemma}\label{lem2.2}
Let $\nu, \sigma$ be a pair of probability Borel measures on the real line such that
$$
\|\nu_h-\nu\|_{\rm TV}\le C_\nu|h|^\alpha,
\quad \|\sigma_h-\sigma\|_{\rm TV}\le C_\sigma|h|^\alpha
$$
for some number $\alpha>0$.
Then
$$
\|\sigma-\nu\|_{\rm TV}\le C(\nu,\sigma)\|\sigma - \nu\|_{\rm FM}^{\frac{\alpha}{1+\alpha}},
$$
where
$$
 C(\nu,\sigma)=2+(C_\sigma+C_\nu)(2\pi)^{-1/2}\int e^{-\frac{t^2}{2}}|t|^\alpha dt.
 $$
\end{lemma}

\begin{lemma}\label{lem2.3}
Let $\nu$ be a Borel probability  measure on the real line.
Assume that for every function $\varphi\in C_b^\infty(\mathbb{R})$ with $\|\varphi\|_\infty\le1$
one has
$$
\int\varphi'd\nu\le C\|\varphi'\|_\infty^{1-\alpha}.
$$
Then, for every Borel set $A$, the following estimate holds:
$$
\nu(A)\le C\lambda(A)^\alpha,
$$
where $\lambda$ is the standard Lebesgue measure on the real line.
Moreover, if $\rho_\nu$ is the density of  $\nu$,
then $\rho_\nu\in L^p(\lambda)$ whenever $1<p<\frac{1}{1-\alpha}$ and one has
$$
\|\rho_\nu\|_{L^p(\lambda)}\le \biggl(p(p-1)^{-1}
+ p\Bigl(\frac{1}{1-\alpha}-p\Bigr)^{-1}\biggr)^{1/p}C^{\frac{1}{\alpha}(1-1/p)}.
$$
\end{lemma}
\begin{proof}
For any $\varphi\in C_0^\infty(\mathbb{R})$ with
$\|\varphi\|_\infty \le1$ we have
$$
\int\varphi d\nu = \int \Phi'd\nu,
$$
where
$$
\Phi(s) := \int_{-\infty}^s \varphi(t)dt.
$$
Note that
$$
|\Phi(s)|\le\int|\varphi(t)|dt.
$$
By the assumptions of the lemma
$$
\biggl(\int|\varphi(t)|dt\biggr)^{-1}
\int \Phi'(s) \nu(ds)\le C\biggl(\int|\varphi(t)|dt\biggr)^{-1+\alpha},
$$
so
$$
\int\varphi d\nu \le C \biggl(\int|\varphi(t)|dt\biggr)^{\alpha}.
$$
Now the first estimate of the lemma can be obtained by approximation.
Let us prove the second part.
Note that
$$
t\lambda(\rho_\nu\ge t)\le \int_{\{\rho_\nu\ge t\}}\rho_\nu d\lambda\le 1.
$$
Applying the first part of the lemma we get
$$
t\lambda(\rho_\nu\ge t)
\le\int_{\{\rho_\nu\ge t\}}\rho_\nu d\lambda = \nu(\rho_\nu\ge t)
\le C\lambda(\rho_\nu\ge t)^\alpha,
$$
therefore,
$$
\lambda(\rho_\nu\ge t)\le t^{-\frac{1}{1-\alpha}}C^{\frac{1}{1-\alpha}}.
$$
Using these estimates and the Fubini theorem, we obtain
\begin{multline*}
\int\rho_\nu^pd\lambda=
p\int_0^\infty t^{p-1}\lambda(\rho_\nu\ge t)dt=
p\biggl(\int_0^\tau + \int_\tau^\infty\biggr)t^{p-1}\lambda(rho_\nu\ge t)dt\\
\le p\int_0^\tau t^{p-2}dt + C^{\frac{1}{1-\alpha}}p
\int_\tau^\infty t^{p-1-\frac{1}{1-\alpha}}dt=
p(p-1)^{-1}\tau^{p-1} + C^{\frac{1}{1-\alpha}}p\Bigl(\frac{1}{1-\alpha}-p\Bigr)^{-1}\tau^{p-\frac{1}{1-\alpha}}.
\end{multline*}
Taking now $\tau = C^{1/\alpha}$, we get
$$
\int\rho_\nu^pd\lambda\le
\Bigl(p(p-1)^{-1} + p\Bigl(\frac{1}{1-\alpha}-p\Bigr)^{-1}\Bigr)C^{(p-1)/\alpha}.
$$
The lemma is proved.
\end{proof}

\begin{remark}\label{rem2.2}
{\rm
Similarly to the previous lemma one can prove that for a function $f\in L^1(\lambda)$ such that
for every function $\varphi\in C_b^\infty(\mathbb{R})$ with $\|\varphi\|_\infty\le1$
$$
\int\varphi' fd\lambda \le C\|\varphi'\|_\infty^{1-\alpha}
$$
one has
$$
\|f\|_{L^p(\lambda)}\le \biggl(p(p-1)^{-1} + p\Bigl(\frac{1}{1-\alpha}-p\Bigr)^{-1}\biggr)^{1/p}\|f\|_{L^1(\lambda)}^{1- \frac{1}{\alpha}(1-1/p)}C^{\frac{1}{\alpha}(1-1/p)}
$$
whenever $1<p<\frac{1}{1-\alpha}$.
Thus,
if $\nu, \sigma$ is a pair of Borel
probability measures on the real line such that for every function
$\varphi\in C_b^\infty(\mathbb{R})$ with $\|\varphi\|_\infty\le1$
one has
$$
\int\varphi'd\nu\le C_\nu\|\varphi'\|_\infty^{1-\alpha}, \quad \int\varphi'd\sigma\le C_\sigma\|\varphi'\|_\infty^{1-\alpha}
$$
and $\rho_\nu, \rho_\sigma$ are their densities, then
$$
\|\rho_\nu - \rho_\sigma\|_{L^p(\lambda)} \le
\biggl(p(p-1)^{-1} + p\Bigl(\frac{1}{1-\alpha}-p\Bigr)^{-1}\biggr)^{1/p}\|\sigma-\nu\|_{\rm TV}^{1- \frac{1}{\alpha}(1-1/p)}(C_\nu + C_\sigma)^{\frac{1}{\alpha}(1-1/p)}
$$
for $1<p<\frac{1}{1-\alpha}$.
}
\end{remark}

\section{Estimates for the Skorohod derivatives of isotropic log-concave measures}

As it has already been said in the introduction, here we obtain an estimate
on the variation of the Skorohod derivative of an isotropic log-concave
measure in terms of its isotropic constant.
The idea is quite simple. If a log-concave measure can be viewed
as a generalization of the uniform distribution on a convex compact set,
then the norm of the Skorohod derivative of this measure can be viewed as a generalization
of the volume of the projection of the convex set.
To estimate the volume of the projection we can embed our convex set into some ball.
Thus, to estimate
the Skorohod derivative of the isotropic log-concave measure
on $\mathbb{R}^n$ with a density $\rho$ and the unit isotropic constant we will
find constants $c_n>0$ and $\alpha_n>0$ such that
$$
\rho(x)\le c_n e^{-\alpha_n|x|}.
$$
Throughout this section we assume that  $\tau\in(0, \infty)$.

\begin{lemma}\label{lem3.1}
Let $\mu$ be a log-concave measure on $\mathbb{R}^n$
with a density $\rho$.
Let $m_\rho = \max \rho$ and let
$$
K=\{x\in\mathbb{R}^n: \rho(x)\ge e^{-\tau}m_\rho\}.
$$
Then for the volume of the body $K$ the following estimate holds true:
$$
1\le m_\rho c_n(\tau)|K|,
$$
where
$\displaystyle c_n(\tau)=n\int_1^{\infty}t^{n-1}e^{-\tau t}dt+1$.
\end{lemma}
\begin{proof}
Without loss of generality we can assume that
$m_\rho = \rho(0)$.
Since $\mu$ is a probability measure, we have
$$
1=\int\rho(x)dx=\int_{\mathbb{R}^n\setminus K}\rho(x)dx+\int_K\rho(x)dx.
$$
By using  polar coordinates we can estimate the first integral in this sum as follows:
\begin{multline*}
\int_{\mathbb{R}^n\setminus K}\rho(x)dx=
\int_{S^{n-1}}\int_{\|\varphi\|_K^{-1}}^{\infty}r^{n-1}\rho(r\varphi)dr\sigma_{n-1}(d\varphi)\\=
\int_{S^{n-1}}\|\varphi\|_K^{-n}\int_1^{\infty}t^{n-1}\rho(t\|\varphi\|_K^{-1}\varphi)dt\sigma_{n-1}(d\varphi).
\end{multline*}
Since $\rho(x)=e^{-V(x)}$ with a convex function $V$, we have
$$
V(t\|\varphi\|_K^{-1}\varphi)\ge tV(\|\varphi\|_K^{-1}\varphi)+(1-t)V(0)
$$
whenever $t\ge 1$.
So,
$$
\rho(t\|\varphi\|_K^{-1}\varphi)\le \rho(0)\Bigl(\frac{\rho(\|\varphi\|_K^{-1}\varphi)}{\rho(0)}\Bigr)^t=\rho(0)e^{-\tau t}.
$$
Thus,
\begin{multline*}
\int_{S^{n-1}}\|\varphi\|_K^{-n}\int_1^{\infty}t^{n-1}\rho(t\|\varphi\|_K^{-1}\varphi)dt\sigma_{n-1}(d\varphi) \le
\rho(0)\int_{S^{n-1}}\|\varphi\|_K^{-n}\int_1^{\infty}t^{n-1}e^{-\tau t}dt\sigma_{n-1}(d\varphi)\\=
\rho(0)(c_n(\tau)-1)n^{-1}\int_{S^{n-1}}\|\varphi\|_K^{-n}\sigma_{n-1}(d\varphi)\\=
\rho(0)(c_n(\tau)-1)\int_{S^{n-1}}\int_0^{\|\varphi\|_K^{-1}}r^{n-1}dr\sigma_{n-1}(d\varphi)=
\rho(0)(c_n(\tau)-1)|K|.
\end{multline*}
Hence,
$1\le\rho(0)c_n(\tau)|K|$,
and the lemma is proved.
\end{proof}

The proof of the following lemma is a combination of the proof of Theorem
4.1 from \cite{KLS} and the previous lemma.

\begin{lemma}\label{lem3.2}
Let $\mu$ be an isotropic log-concave measure on $\mathbb{R}^n$
with a density $\rho$ and the isotropic constant $1$.
Let $m_\rho = \max \rho$ and let
$$
K=\{x\in\mathbb{R}^n\colon\ \rho(x)\ge e^{-\tau}m_\rho\}.
$$
Then for every point $x\in K$ we have
$$
|x|^2\le c_n(\tau)(n+1)^2e^\tau,
$$
where
$\displaystyle c_n(\tau)=n\int_1^{\infty}t^{n-1}e^{-\tau t}dt+1$.
\end{lemma}
\begin{proof}
Let $v\in K$.
Note that
$$
\int_K(x, \theta)^2dx=
\int_{K-v}(v+u, \theta)^2du=
\int_{S^{n-1}}\int_0^{\|\varphi\|_{K-v}^{-1}}r^{n-1} (v+r\varphi, \theta)^2dr\sigma_{n-1}(d\varphi),
$$
where in the second equality we pass to polar coordinates.
Calculating the inner integral in $r$ we have
\begin{multline*}
\int_{S^{n-1}}(n^{-1}\|\varphi\|_{K-v}^{-n}(v,\theta)^2+
2(n+1)^{-1}\|\varphi\|_{K-v}^{-n-1}(v,\theta)(\varphi,\theta)+
(n+2)^{-1}\|\varphi\|_{K-v}^{-n-2}(\varphi,\theta)^2)\sigma_{n-1}(d\varphi)\\=
\int_{S^{n-1}}\|\varphi\|_{K-v}^{-n}n^{-1}\Bigl(\frac{\sqrt{n(n+2)}}{n+1}(v,\theta)+
\sqrt{\frac{n}{n+2}}\|\varphi\|_{K-v}^{-1}(\varphi,\theta)\Bigr)^2\\+
n^{-1}(n+1)^{-2}\|\varphi\|_{K-v}^{-n}(v,\theta)^2\sigma_{n-1}(d\varphi)\ge
n^{-1}(n+1)^{-2}(v,\theta)^2\int_{S^{n-1}}\|\varphi\|_{K-v}^{-n}\sigma_{n-1}(d\varphi),
\end{multline*}
where the last inequality is valid, since
 the first term under the integral sign is nonnegative.
In turn, the last expression is equal to
$$
(n+1)^{-2}(v,\theta)^2\int_{S^{n-1}}\int_0^{\|\varphi\|_{K-v}^{-1}}r^{n-1}dr\sigma_{n-1}(d\varphi)=
(n+1)^{-2}(v,\theta)^2|K|.
$$
Thus, using the estimate for the volume of the body $K$ from the previous lemma,
we obtain
\begin{multline*}
|\theta|^2 =\int(x, \theta)^2\rho(x)dx
\ge e^{-\tau}m_\rho\int_K(x, \theta)^2dx
\\
\ge
e^{-\tau}m_\rho(n+1)^{-2}|K|(v,\theta)^2\ge
\frac{e^{-\tau}}{c_n(\tau)(n+1)^2}(v,\theta)^2.
\end{multline*}
Now we can take $\theta = v$ and get the estimate $|v|^2\le c_n(\tau)(n+1)^2e^\tau$.
\end{proof}

\begin{lemma}\label{lem3.3}
For every $n\in\mathbb{N}$,
there are constants $c_n>0$ and $\alpha_n>0$ depending only on  $n$
 such that, for every
isotropic log-concave measure $\mu$ on $\mathbb{R}^n$
with a density $\rho$ and the isotropic constant $1$,
the following inequality holds true:
$$
\rho(x)\le c_n e^{-\alpha_n|x|}.
$$
\end{lemma}
\begin{proof}
Let $x_0$ be  a point such that $\rho(x_0) = \max \rho = m_\rho$,
$$
K=\{x\in\mathbb{R}^n\colon
 \rho(x)\ge e^{-\tau}m_\rho\},\quad r = (n+1)\sqrt{c_n(\tau)e^\tau},
$$
where $c_n(\tau)$ is the constant from the previous lemma.
By the previous lemma we have the inclusion $K\subset B_r$,
where $B_r$ is the ball of radius $r$ centered at the origin.

Suppose first that $x\not\in B_{2r}$, equivalently, $\|x-x_0\|_{B_{2r}-x_0}\ge1$,
where $\|\cdot\|_{B_{2r}-x_0}$ is the Minkowski functional of the set $B_{2r}-x_0$.
Note that
$$
x = (1-\|x-x_0\|_{B_{2r}-x_0})x_0 + \|x-x_0\|_{B_{2r}-x_0}\Bigl(x_0+\frac{x-x_0}{\|x-x_0\|_{B_{2r}-x_0}}\Bigr).
$$
Since $\mu$ is a log-concave measure, the density $\rho$ is of the form $e^{-V}$
with some convex function $V$.
Using the equality mentioned above and the convexity of the function $V$ we have
\begin{multline*}
V(x)\ge (1-\|x-x_0\|_{B_{2r}-x_0})V(x_0) + \|x-x_0\|_{B_{2r}-x_0}V
\Bigl(x_0+\frac{x-x_0}{\|x-x_0\|_{B_{2r}-x_0}}\Bigr)\\ \ge
(1-\|x-x_0\|_{B_{2r}-x_0})V(x_0) + \|x-x_0\|_{B_{2r}-x_0}(V(x_0)+\tau)
=V(x_0) + \tau \|x-x_0\|_{B_{2r}-x_0},
\end{multline*}
where the inequality follows from the fact that the point
$x_0+\frac{x-x_0}{\|x-x_0\|_{B_{2r}-x_0}}$ does not belong to the set $K$.

We now estimate $\|x-x_0\|_{B_{2r}-x_0}$.
Note that
$$
\Bigl|\frac{x-x_0}{\|x-x_0\|_{B_{2r}-x_0}}+x_0\Bigr|=2r,
$$
hence
$$
2r\|x-x_0\|_{B_{2r}-x_0}\ge |x| - (\|x-x_0\|_{B_{2r}-x_0}-1)|x_0|\ge |x| - (\|x-x_0\|_{B_{2r}-x_0}-1)r,
$$
which implies the estimate
$$
\|x-x_0\|_{B_{2r}-x_0}\ge1/3 + (3r)^{-1}|x|.
$$
Thus, we have
$$
V(x)\ge V(x_0) + \tau (1/3 + (3r)^{-1}|x|),
$$
hence,
$$
\rho(x) \le e^{-\tau/3}m_\rho e^{-\tau(3r)^{-1}|x|}\le m_\rho e^{-\tau(3r)^{-1}|x|}\le e^\tau m_\rho e^{-\tau(3r)^{-1}|x|}.
$$

We now consider the case $x\in B_{2r}$.
In this case we have
$$
\rho(x)\le m_\rho \le e^\tau m_\rho e^{-\tau(3r)^{-1}|x|}.
$$
By Theorem \ref{t1.1} one has $m_\rho\le C_n^n$
for some constant $C_n$, so, taking
$\tau = 1$, we obtain the desired estimate.
\end{proof}

\begin{theorem}\label{t3.1}
Let $\mu$ be an isotropic log-concave measure on $\mathbb{R}^n$
with a density $\rho$ and the isotropic constant $1$.
Let $h$ be a vector of unit length. Then
$$
\|D_h\mu\|\le C(n)
$$
with some constant $C(n)$ that depends only on the dimension $n$.
\end{theorem}
\begin{proof}
Without loss of generality we can assume that $h=e_1$ is the first basis vector.
According to Krugova's result \cite{Krug},
$$
\frac{1}{2}\|D_h\mu\|=\int_{\mathbb{R}^{n-1}}
\max_t\rho(t,x_2,\ldots, x_n)dx_2\ldots dx_n.
$$
Now we can use the estimate from Lemma \ref{lem3.3}:
$$
\int_{\mathbb{R}^{n-1}}\max\limits_t\rho(t,x_2,\ldots, x_n)dx_2\ldots dx_n\le
c_n\int_{\mathbb{R}^{n-1}}\exp\Bigl(-\alpha_n\Bigl(\sum_{i=2}^nx_i^2\Bigr)^{1/2}\Bigr)
dx_2\ldots dx_n,
$$
where the last expression depends only on the dimension $n$, and the lemma is proved.
\end{proof}

\section{Verification of the fractional smoothness sufficient condition}

This section is devoted to obtaining a technical statement in Theorem \ref{t4.1}.
First, using estimates of the Skorohod derivatives from the previous section,
the Poincar\'e inequality, and the Carbery--Wright inequality we
obtain an estimate that depends on dimension. Then we use the localization techniques
to make this estimate dimension-free.

\begin{lemma}\label{lem4.1}
Let $\mu$ be a log-concave measure on $\mathbb{R}^n$ with a density $\rho$
with respect to Lebesgue measure.
Then, for every $d\in\mathbb{N}$,
 there are constants $c_1(d), c_2(d)$ depending only on $d$ such that,
for every polynomial $f$  of degree $d$ on $\mathbb{R}^n$,
every function $\varphi\in C_b^\infty(\mathbb{R})$ with
$\|\varphi\|_\infty\le1$, and for every vector~$e$
of unit length the following inequality holds true:
$$
\int\varphi'(f)d\mu\le \Bigl(c_1(d)\|\partial_ef\|_2^{-1/(d-1)}+c_2(d)\|D_e\mu\|_{\rm TV}\Bigr)\|\varphi'\|_\infty^{1-1/d}.
$$
\end{lemma}
\begin{proof}
Fix a number $\varepsilon>0$.
Note that
$$
\int\varphi'(f)d\mu=\int\frac{(\partial_ef)^2}{(\partial_ef)^2+\varepsilon}\varphi'(f)d\mu+
\varepsilon\int\varphi'(f)((\partial_ef)^2+\varepsilon)^{-1}d\mu.
$$
Consider the first term:
$$
\int\frac{(\partial_ef)^2}{(\partial_ef)^2+\varepsilon}\varphi'(f)d\mu=
\int\partial_e(\varphi(f))\frac{\partial_ef}{(\partial_ef)^2+\varepsilon}d\mu.
$$
Applying the integration by parts formula to the last expression
(recall that $D_e\mu$ denotes the Skorohod derivative of the measure $\mu$
along the vector $e$) we have
\begin{multline*}
-\int\varphi(f)\Bigl[\frac{\partial^2_ef}{(\partial_ef)^2+\varepsilon}-
2\frac{(\partial_ef)^2\partial^2_ef}{((\partial_ef)^2+\varepsilon)^2}\Bigr]d\mu -
\int\varphi(f)\frac{\partial_ef}{(\partial_ef)^2+\varepsilon}d(D_e\mu)\\ \le
\int\Bigl|\frac{\partial^2_ef}{(\partial_ef)^2+\varepsilon}\Bigr|d\mu+
2\int\Bigl|\frac{(\partial_ef)^2\partial^2_ef}{((\partial_ef)^2+\varepsilon)^2}\Bigr|d\mu+
\int\Bigl|\frac{\partial_ef}{(\partial_ef)^2+\varepsilon}\Bigr|d|D_e\mu|\\=
\varepsilon^{-1/2}\biggl(\int\Bigl|\frac{\partial^2_eg}{(\partial_eg)^2+1}\Bigr|d\mu+
2\int\Bigl|\frac{(\partial_eg)^2\partial^2_eg}{((\partial_eg)^2+1)^2}\Bigr|d\mu+
\int\Bigl|\frac{\partial_eg}{(\partial_eg)^2+1}\Bigr|d|D_e\mu|\biggr),
\end{multline*}
where $g=f\varepsilon^{-1/2}$. Now let us estimate each term separately:
$$
\int\Bigl|\frac{\partial_eg}{(\partial_eg)^2+1}\Bigr|d|D_e\mu|\le1/2\|D_e\mu\|_{\rm TV},
$$
$$
\int\Bigl|\frac{\partial^2_eg}{(\partial_eg)^2+1}\Bigr|d\mu=
\int_{\langle e\rangle^\bot}\int_{\langle e\rangle}\Bigl|\frac{\partial^2_eg(x+te)}{(\partial_eg(x+te))^2+1}\Bigr|\rho(x+te)dtdx
$$
$$
\le d\int_{\langle e\rangle^\bot}\max\limits_{s}
\rho(x+se)\int_{\langle e\rangle}\Bigl|\frac{1}{\tau^2+1}\Bigr|d\tau dx=1/2d\pi\|D_e\mu\|_{\rm TV},
$$
$$
\int\Bigl|\frac{(\partial_eg)^2\partial^2_eg}{((\partial_eg)^2+1)^2}\Bigr|d\mu\le
\int\Bigl|\frac{\partial^2_eg}{(\partial_eg)^2+1}\Bigr|d\mu\le 1/2d\pi\|D_e\mu\|_{TV}.
$$
Thus,
$$
\int\frac{(\partial_ef)^2}{(\partial_ef)^2+\varepsilon}\varphi'(f)d\mu \le \varepsilon^{-1/2}2^{-1}(1+3d\pi)\|D_e\mu\|_{\rm TV}.
$$
Let $M=\|\varphi'\|_\infty$. Let us estimate the expression
$$
M^{-1}\int\varphi'(f)((\partial_ef)^2+\varepsilon)^{-1}d\mu.
$$
This expression can be estimated from above by
\begin{multline*}
\int((\partial_ef)^2+\varepsilon)^{-1}d\mu=\int_0^{1/\varepsilon}\mu(((\partial_ef)^2+\varepsilon)^{-1}>t)dt=
\int_0^{1/\varepsilon}\mu((\partial_ef)^2<1/t-\varepsilon)dt\\
=\int_0^\infty(s+\varepsilon)^{-2}\mu((\partial_ef)^2<s)ds\le
c_1d\|\partial_ef\|_2^{-1/(d-1)}\int_0^\infty(s+\varepsilon)^{-2}s^{1/(2d-2)}ds\\
=c_1d\int_0^\infty(s+1)^{-2}s^{1/(2d-2)}ds\|\partial_ef\|_2^{-1/(d-1)}\varepsilon^{-1+1/(2d-2)},
\end{multline*}
where the Carbery--Wright inequality from Theorem \ref{t1.3} was used in the last step.
Hence
$$
\int\varphi'(f)d\mu\le c_1(d)\|\partial_ef\|_2^{-1/(d-1)}M\varepsilon^{1/(2d-2)}+c_2(d)\|D_e\mu\|_{\rm TV}\varepsilon^{-1/2},
$$
where
$$
c_1(d)=cd\int_0^\infty(s+1)^{-2}s^{1/(2d-2)}ds,\quad c_2(d)=2^{-1}(1+3d\pi).
$$
Set $\varepsilon=M^{-2+2/d}$. Then
$$
\int\varphi'(f)d\mu\le\Bigl(c_1(d)\|\partial_ef\|_2^{-1/(d-1)}+c_2(d)\|D_e\mu\|_{TV}\Bigr)M^{1-1/d}.
$$
The lemma is proved.
\end{proof}

\begin{corollary}\label{c4.1}
Let $n, d\in\mathbb{N}$.
Then there is a constant $c(d,n)$ depending only on
$d$ and $n$ such that, whenever
$\mu$ is a log-concave measure on $\mathbb{R}^n$
with a density $\rho$,   $f$ is a
 polynomial of degree $d$ on $\mathbb{R}^n$,
for every function
$\varphi\in C_b^\infty(\mathbb{R})$ with $\|\varphi\|_\infty\le1$
the following inequality holds:
$$
\int\varphi'(f)d\mu\le
c(d,n)\biggl(\biggl(\int|\nabla f|^{1/(d-1)}d\mu\biggr)^{-1}+\max\limits_{|e|=1}\|D_e\mu\|_{TV}\biggr)\|\varphi'\|_\infty^{1-1/d}.
$$
\end{corollary}
\begin{proof}
Multiplying the inequality from Lemma \ref{lem4.1} by $\int|(\nabla f, e)|^{1/(d-1)}d\mu$
and using the inequality
$$
\int|(\nabla f, e)|^{1/(d-1)}d\mu\le\biggl(\int|(\nabla f, e)|^2d\mu\biggr)^{1/(2d-2)},
$$
we get the estimate
$$
\int|(\nabla f, e)|^{1/(d-1)}d\mu\int\varphi'(f)d\mu\le
\biggl(c_1(d)+c_2(d)\|D_e\mu\|_{TV}\int|(\nabla f, e)|^{1/(d-1)}d\mu\biggr)\|\varphi'\|_\infty^{1-1/d}.
$$
Estimating $\|D_e\mu\|_{TV}$ by the maximum on the sphere
and integrating with respect to the standard normalized surface
 measure $\sigma_n$ on the sphere, we obtain
\begin{multline*}
\int_{S^{n-1}}\int|(\nabla f, e)|^{1/(d-1)}d\mu\sigma_n(de)\int\varphi'(f)d\mu\\ \le
\biggl(c_1(d)+
c_2(d)\max\limits_{|e|=1}\|D_e\mu\|_{TV}\int_{S^{n-1}}\int|(\nabla f, e)|^{1/(d-1)}d\mu\sigma_n(de)\biggr)\|\varphi'\|_\infty^{1-1/d}.
\end{multline*}
Note that
\begin{multline*}
\int_{S^{n-1}}\int|(\nabla f, e)|^{1/(d-1)}d\mu\sigma_n(de)=
\int\int_{S^{n-1}}|(\nabla f, e)|^{1/(d-1)}\sigma_n(de)d\mu\\=
\int|\nabla f|^{1/(d-1)}\int_{S^{n-1}}|(e,e_1)|^{1/(d-1)}\sigma_n(de)d\mu=
C(n,d)\int|\nabla f|^{1/(d-1)}d\mu,
\end{multline*}
where
$$
C(n,d) = \int_{S^{n-1}}|(e, e_1)|^{1/(d-1)}\sigma_n(de).
$$
Thus, the assertion of the corollary is true with the constant
$$
c(d,n) = \max\Bigl\{\frac{c_1(d)}{C(n,d)} , c_2(d)\Bigr\}.
$$
The corollary is proved.
\end{proof}

\begin{corollary}\label{c4.2}
Let $d,n\in \mathbb{N}$.
Then, there is a constant $C(d,n)$
depending only on  $d$ and $n$ such that,
whenever $\mu$ is an isotropic log-concave measure on $\mathbb{R}^n$
with a density $\rho$ and the unit isotropic constant and
  $f$ is a  polynomial of degree $d$ on $\mathbb{R}^n$,
for every function $\varphi\in C_b^\infty(\mathbb{R})$ with $\|\varphi\|_\infty\le1$
one has
$$
\int\varphi'(f)d\mu\le
C(d,n)\biggl(\Bigl(\int\bigl|f-\mathbb{E}f\bigr|^{1/(d-1)}d\mu\Bigr)^{-1}+1\biggr)\|\varphi'\|_\infty^{1-1/d}.
$$
\end{corollary}
\begin{proof}
By Theorem \ref{t1.4} for the log-concave measure $\mu$ we have
$$
\int\bigl(f-\mathbb{E}f\bigr)^2d\mu\le R\int|x-x_0|^2d\mu\int|\nabla f|^2d\mu,
$$
where $R$ is an absolute constant, $x_0=\int x d\mu$.
Note that
$$
\int|x-x_0|^2d\mu=\int|x|^2d\mu=n.
$$
By Theorem \ref{t3.1} we have
$\|D_e\mu\|\le C(n)$.
Using the estimates from Theorem \ref{t1.2} and Corollary~\ref{c4.1},
we obtain the desired inequality.
\end{proof}

\begin{theorem}\label{t4.1}
Let $d\ge2$.
Then, there is a constant $C(d)$ depending only on  $d$ such that, whenever
$\mu$ is a log-concave measure on $\mathbb{R}^n$,
  $f$ is a polynomial of degree $d$ on $\mathbb{R}^n$,
for every function $\varphi\in C_b^\infty(\mathbb{R})$ with $\|\varphi\|_\infty\le1$
one has
$$
\sigma_f^{1/d}\int\varphi'(f)d\mu\le
C(d)\|\varphi'\|_\infty^{1-1/d}.
$$
\end{theorem}
\begin{proof}
First, we prove the estimate
\begin{equation}\label{4.1}
\int\varphi'(f)d\mu\le
C_1(d)\biggl(\Bigl(\int\bigl|f-\mathbb{E}f\bigr|^{1/(d-1)}d\mu\Bigr)^{-1} + 1\biggr)\|\varphi'\|_\infty^{1-1/d}.
\end{equation}
Note that it is sufficient to prove this estimate only for log-concave measures with compact support.
Let us fix
a log-concave measure $\mu$ with a convex compact support $K$ and a polynomial $f$ of degree~$d$.
Without loss of generality we can assume that $\mathbb{E}f:=\int f d\mu =0$.
Let $k := \int |f|^{1/(d-1)}d\mu$.
Let $P_f$ be the set of all log-concave measures $\nu$ with support in $K$ such that
$$
\int fd\nu\ge0, \quad \int -fd\nu\ge0, \int |f|^{1/(d-1)}d\nu\ge k.
$$
Note that $\mu\in P_f$. Consider the function
$$
\Phi_f(\nu) := \int\varphi'(f)d\nu -
C_1(d)\bigl(k^{-1}+1\bigr)\|\varphi'\|_\infty^{1-1/d},
$$
where $C_1(d)=\max\{C(d,n), n=1, 2, 3\}$, $C(d,n)$ is the constant from the previous corollary.
Now we prove that $\Phi_f(\nu)\le0$ for
every polynomial $f$ of degree $d$ and every log-concave measure $\nu\in P_f$.
Due to Theorem \ref{t1.5} (with $p=3$, $f_1=f$, $f_2=-f$, $f_3=|f|^{1/(d-1)}$)
it is sufficient to prove this inequality for every polynomial $f$ of degree $d$ and
every log-concave measure $\nu\in P_f$ such that the smallest
affine subspace containing the support of $\nu$
is of dimension not greater than $3$.
For every log-concave measure $\nu$ with a density,
there is a non-degenerate linear transformation $T$ such that
the measure $\nu\circ T^{-1}$ is isotropic with the unit isotropic constant.
Since a polynomial of degree $d$ composed with a linear transformation
is a polynomial of degree~$d$, it is sufficient to prove our inequality
for every isotropic log-concave measure $\nu$
with the unit isotropic constant in dimensions at most $3$
and every polynomial $f$ of degree $d$.
By Corollary \ref{c4.2} one has $\Phi_f(\nu)\le0$
for such measures and polynomials.
Thus, $\Phi_f(\nu)\le0$ for every polynomial $f$ of degree $d$
and every log-concave measure $\nu\in P_f$,
and since $\mu\in P_f$ and $k=\int |f|^{1/(d-1)}d\mu$, we have (\ref{4.1}).
By Theorem \ref{t1.2}
$$
\int\bigl|f-\mathbb{E}f\bigr|^{1/(d-1)}d\mu\ge (4c(d-1))^{-1}\biggl(\int\bigl|f-\mathbb{E}f\bigr|^2d\mu\biggr)^{1/(2d-2)}.
$$
Let $f$ be a polynomial of degree $d$ on $\mathbb{R}^n$.
Then by inequality (\ref{4.1}) we have
$$
\int\varphi'(f\sigma_f^{-1})d\mu\le
C_1(d)\bigl(4c(d-1) + 1\bigr)\|\varphi'\|_\infty^{1-1/d}.
$$
Let $\psi(t)=\varphi(t\sigma_f^{-1})$, $C(d)=C_1(d)(4c(d-1) + 1)$. Then
$$
\int\psi'(f)d\mu=\sigma_f^{-1}\int\varphi'(f\sigma_f^{-1})d\mu\le
C(d)\sigma_f^{-1}\|\varphi'\|_\infty^{1-1/d}=
C(d)\sigma_f^{-1/d}\|\psi'\|_\infty^{1-1/d}.
$$
Thus, the theorem is now proved.
\end{proof}

\section{Properties of polynomial images of log-concave measures}

Here we use an
approximation argument, Theorem \ref{t4.1}, and the results of Section~2 to obtain
the main results of this work:
Corollaries \ref{c5.1}, \ref{c5.2}, \ref{c5.3}, and \ref{c5.4}.
In the formulations below by non-constant functions we mean functions
that do not coincide with constants almost everywhere.

\begin{theorem}\label{t5.1}
Let $d\in\mathbb{N}$.
Then, there is a constant $C(d)$ depending only on  $d$
such that, whenever
$\mu$ is a Radon log-concave measure on a locally convex space $E$ and
$f\in\mathcal{P}^d(\mu)$, for
every function $\varphi\in C_b^\infty(\mathbb{R})$ with $\|\varphi\|_\infty\le1$
one has
\begin{equation}\label{5.1}
\sigma_f^{1/d}\int\varphi'(f)d\mu\le
C(d)\|\varphi'\|_\infty^{1-1/d}.
\end{equation}
\end{theorem}
\begin{proof}
Let $f_n = f_n(\ell_{1,n},\ldots, \ell_{k_n,n})$
be a sequence of polynomials of degree $d$ in finitely many variables such
that its limit in $L^2(\mu)$ is $f$.
The polynomials $f_n$ satisfy inequality (\ref{5.1}) by Theorem \ref{t4.1}.
Since
$$
\int\varphi'(f_n)d\mu\to\int\varphi'(f)d\mu,
$$
$$
\int\bigl(f_n-\mathbb{E}f_n\bigr)^2d\mu\to\int\bigl(f-\mathbb{E}f\bigr)^2d\mu,
$$
this inequality is also valid for the function $f$.
\end{proof}

\begin{corollary}\label{c5.1}
Let $\mu$ be a Radon log-concave measure on a locally convex space $E$.
Let $C(d)$ be the same constant as in Theorem \ref{t5.1}
and let $f\in\mathcal{P}^d(\mu)$.
Then
$$
\sigma_f^{1/d}
\|(\mu\circ f^{-1})_h-\mu\circ f^{-1}\|_{TV}\le 2^{1-1/d}C(d)|h|^{1/d},
$$
$$
\sigma_f^{1/d}\mu\circ f^{-1}(A)\le C(d)\lambda(A)^{1/d},
$$
where $\lambda$ is the standard Lebesgue measure on the real line, $A$ is a
 Borel set on the real line.
In particular, the measure $\mu\circ f^{-1}$ is absolutely continuous.
\end{corollary}
\begin{proof}
We apply  Lemma \ref{lem2.1} and  Lemma \ref{lem2.3}.
\end{proof}

Denote by $\rho_f$ the density of the measure $\mu\circ f^{-1}$,
which exists by the previous corollary.
The next corollary follows from Lemma \ref{lem2.3}.

\begin{corollary}\label{c5.2}
Let $\mu$ be a Radon log-concave measure on a locally convex space $E$ and let
$f\in\mathcal{P}^d(\mu)$.
Then $\rho_f\in L^p(\mathbb{R})$ whenever $1<p< \frac{d}{d-1}$,
and for the $L^p$-norm of $\rho_f$ the following inequality holds:
$$
\sigma_f^{1-1/p}\|\rho_f\|_{L^p(\mathbb{R})}\le C_1(d,p),
$$
where
$$
C_1(d,p) = \biggl(p(p-1)^{-1} +p\Bigl(\frac{d}{d-1}-p\Bigr)^{-1}\biggr)^{1/p}C(d)^{d(1-1/p)}.
$$
\end{corollary}

Our next corollary generalizes some
results from \cite{Nourdin2, Nourdin3} and \cite{BKZ}
to the case of  arbitrary log-concave measures in place of Gaussian measures,
and even in the Gaussian case this estimate provides a better rate of convergence
as  compared to analogous estimates
from \cite{Nourdin2} and \cite{BKZ} in the one-dimensional case.
It follows directly from Lemma \ref{lem2.2} and Theorem \ref{t5.1}.

\begin{corollary}\label{c5.3}
Let $\mu$ be a Radon log-concave measure on a locally convex space $E$.
Then for every pair of non-constant functions $f, g\in \mathcal{P}^d(\mu)$
one has
$$
\|\mu\circ f^{-1}-\mu\circ g^{-1}\|_{\rm TV}\le
C_d(\sigma_f,\sigma_g)\|\mu\circ f^{-1} - \mu\circ g^{-1}\|_{\rm FM}^{1/(1+d)},
$$
where
$$
C_d(\sigma_f,\sigma_g)=1+2C(d)(\sigma_f^{-1/d} +
\sigma_g^{-1/d})
(2\pi)^{-1/2}\int e^{-t^2/2}|t|^{1/d} dt.
$$
\end{corollary}

Combining Remark \ref{rem2.2} with the
previous lemma, we obtain the following result.

\begin{corollary}\label{c5.4}
Let $\mu$ be a Radon log-concave measure on a locally convex space $E$.
Then for every pair of non-constant
functions $f, g\in \mathcal{P}^d(\mu)$
one has
\begin{multline*}
\|\rho_f-\rho_g\|_{L^p(\mathbb{R})}\le C_1(d,p) (\sigma_f^{-1/d}+\sigma_g^{-1/d})^{d(1-1/p)}\|\mu\circ f^{-1}-\mu\circ g^{-1}\|_{\rm TV}^{1-d(1-1/p)}\\
\le
C_1(d,p) (\sigma_f^{-1/d}+\sigma_g^{-1/d})^{d(1-1/p)} C_d(\sigma_f,\sigma_g)^{1-d(1-1/p)}\|\mu\circ f^{-1} - \mu\circ g^{-1}\|_{\rm FM}^{\frac{1}{1+d}(1-d(1-1/p))},
\end{multline*}
where $\rho_f, \rho_g$ are densities of the measures
$\mu\circ f^{-1}, \mu\circ g^{-1}$, respectively, and $1<p< \frac{d}{d-1}$.
\end{corollary}

\end{document}